\documentclass[a4paper,12pt]{amsart}
\usepackage{amsfonts}
\usepackage{amsmath}
\usepackage{amsthm}
\usepackage{amssymb}
\usepackage{mathrsfs}
\usepackage{latexsym}
\usepackage{float}
\usepackage{vmargin}
\setmargrb{3cm}{2.5cm}{2.5cm}{3cm} 
\begin{document}

%%%%%%%%%%%%%%%%%%%%%%%%%%%%%%%%%%%%%%%%%%%%%%%%%%%%%%%%%%%%%%%%%%%%%%%%

\newtheorem{theorem}{Theorem}[section]
\newtheorem{definition}[theorem]{Definition}
\newtheorem{lemma}[theorem]{Lemma}
\newtheorem{proposition}[theorem]{Proposition}
\newtheorem{corollary}[theorem]{Corollary}
\newtheorem{example}[theorem]{Example}
\newtheorem{remark}[theorem]{Remark}

%%%%%%%%%%%%%%%%%%%%%%%%%%%%%%%%%%%%%%%%%%%%%%%%%%%%%%%%%%%%%%%%%%%%%%%%
%\def\thefootnote{\fnsymbol{\footnote}}

\pagestyle{plain}
\title{{P}rojections in noncommutative tori and Gabor frames}

\author{Franz Luef} 
\thanks{The author was supported by the Marie Curie Excellence grant MEXT-CT-2004-517154 and the Marie Curie Outgoing Fellowship PIOF-220464.}
\address{Current address: Department of Mathematics\\ University of California\\Berkeley CA 94720-3840}
\email{luef@math.berkeley.edu}
\address{Fakult\"at f\"ur Mathematik\\Universit\"at Wien\\
Nordbergstrasse 15\\ 1090 Wien\\ Austria\\}
\email{franz.luef@univie.ac.at}
%  Math Subject Classifications
%\footnote[1]{The author was supported by the EU-project MEXT-CT-2004-517154.}
%%%%%%%%%%%%%%%%%%%%%%%%%%%%%%%%%%%%%%%%%%%%%%%%%%%%%%%%%%%%%%%%%%%%%%%%%%%%%%%%%%%%%%%%%%%%%%%%%%%%%%%%%%%%%%%%%%%%%%%%%%%%%%%%%%%%%%%%%%%%%%%%%%%%%%%%
\begin{abstract}
We describe a connection between two seemingly different problems: (a) the construction of projections in noncommutative tori, (b) the construction of tight Gabor frames for $L^2(\mathbb{R})$. 
The present investigation relies an interpretation of projective modules over noncommutative tori in terms of Gabor analysis. The 
main result demonstrates that Rieffel's condition on the existence of projections in noncommutative tori is equivalent to the Wexler-Raz biorthogonality relations for tight Gabor frames. Therefore we are able to invoke results on the existence of Gabor frames in the construction of  projections in noncommutative tori. In particular, the projection associated with a Gabor frame generated by a Gaussian turns out to be Boca's projection. Our approach to Boca's projection allows us to characterize the range of existence of Boca's projection. The presentation of our main result provides a natural approach to the Wexler-Raz biorthogonality relations in terms of Hilbert $C^*$-modules over noncommutative tori. 
\end{abstract}
\maketitle \pagestyle{myheadings} \markboth{F. Luef}{Projections in noncommutative tori and Gabor frames}
\thispagestyle{empty}
%%%%%%%%%%%%%%%%%%%%%%%%%%%%%%%%%%%%%%%%%%%%%%%%%%%%%%%%%%%%%%%%%%%%%%%%%%%%%%%%%%%%%%%%%%%%%%%%%%%%%%%%%%%%%%%%%%%%%%%%%%%%%%%%%%%%%%%%%%%%%%%%%%%%%%%%
\section{Introduction}
Projections in $C^*$-algebras and von Neumann algebras are of great relevance for the exploitation of its structures. Von Neumann algebras contain an abundance of projections. The question of existence of projections in a $C^*$-algebra is a non-trivial task and the answer to this question has many important consequences, e.g. for the $K$-theory of $C^*$-algebras. Therefore many contributions to $C^*$-algebras deal with the existence and construction of projections in various classes of $C^*$-algebras. In the present investigation we focus on the construction of projections in noncommutative tori $\mathcal{A}_\theta$ for a real number $\theta$. Recall that $\mathcal{A}_\theta$ is the universal $C^*$-algebra generated by two unitaries $U_1$ and $U_2$ which satisfy the following commutation relation 
\begin{equation}\label{commrelation}
  U_2U_1=e^{2\pi i\theta}U_1U_2.
\end{equation}
In the seminal paper \cite{ri81} Rieffel constructed projections for noncommutative tori $\mathcal{A}_\theta$ with $\theta$ irrational and drew some consequences for the $K$-theory of $\mathcal{A}_\theta$, e.g. that the projections in $\mathcal{A}_\theta$ generate $K_0(\mathcal{A}_\theta)$. 
\par
The main goal of this study is to show that Rieffel's construction of projections in noncommutative tori is intimately related to the existence of Gabor frames for $L^2(\mathbb{R})$. A {\it Gabor system} is a collection of functions $\mathcal{G}(g,\Lambda)=\{\pi(\lambda)g:\lambda\in\Lambda\}$ in $L^2(\mathbb{R})$, where $g$ is a function in $L^2(\mathbb{R})$, $\Lambda$ is a lattice in $\mathbb{R}^2$, and $\pi(\lambda)g$ is the time-frequency shift by $\lambda\in\Lambda$ of $g$. For  $z=(x,\omega)$ in $\mathbb{R}^2$ we denote by $\pi(z)=M_\omega T_x$ the  {\it time-frequency shift}, where $T_x$ denotes the {\it translation operator} $T_xg(t)=g(t-x)$ and $M_\omega$ the {\it modulation operator} $M_{\omega}g(t)=e^{2\pi it\cdot\omega}g(t)$. A Gabor system $\mathcal{G}(g,\Lambda)$ is a  {\it Gabor frame} for $L^2(\mathbb{R})$, if there exist $A,B>0$ such that 
\begin{equation}\label{GaborFrame}
  A\|f\|_2^2\le\sum_{\lambda\in\Lambda}|\langle f,\pi(\lambda)g\rangle|^2\le B\|f\|_2^2
\end{equation} 
holds for all $f\in L^2(\mathbb{R})$. The field of Gabor analysis is a branch of time-frequency analysis that has its origins in the seminal paper \cite{ga46} of the Nobel laureate D. Gabor. We refer the interested reader to \cite{gr01} for an excellent introduction to Gabor analysis.   
\par
Gabor frames $\mathcal{G}(g,\mathbb{Z}\times\theta\mathbb{Z})$ are intimately related with noncommutative tori $\mathcal{A}_\theta$. Namely the two unitaries $U_1=T_{1}$ and $U_2=M_{\theta}$ provide a faithful representation of $\mathcal{A}_\theta$ on $\ell^2(\mathbb{Z}^2)$, because $T_1$ and $M_\theta$ satisfy the commutation relation from Eq. \eqref{commrelation}:
\begin{equation}\label{TM-ComRel}
  M_\theta T_1=e^{2\pi i\theta}T_1 M_\theta.
\end{equation}
The construction of projections in \cite{ri81} relies on the existence of a $C^*$-algebra $\mathcal{B}$ that is Morita-Rieffel equivalent to $\mathcal{A}_\theta$ through an equivalence bimodule $_{\mathcal{A}_\theta} V_\mathcal{B}$. In \cite{co80} and \cite{ri81} Connes and Rieffel determined the class of $C^*$-algebras that are Rieffel-Morita equivalent to $\mathcal{A}_\theta$. Most notably the opposite algebra of $\mathcal{A}_{1/\theta}$ is Morita-Rieffel equivalent to $\mathcal{A}_\theta$. In \cite{lu07,lu09} we were able to link this important result with Gabor analysis, which allows us to interpret Rieffel's condition on the existence of 
projections in $\mathcal{A}_\theta$ as the Wexler-Raz duality biorthogonality relations for tight Gabor frames. 
\par
The Wexler-Raz duality biorthogonality relations were first discussed in the finite-dimensional setting \cite{rawe90}. The extension of the results in \cite{rawe90} to the infinite-dimensional setting was the main impetus of several groups of mathematicians in time-frequency analysis and it led to the development of the duality theory of Gabor analysis \cite{dalala95,ja95,rosh97}. We follow the work of Janssen in \cite{ja95}, since it provides the most natural link to Rieffel's work on projective modules over noncommutative tori \cite{ri88}. 
\par
The projections in $\mathcal{A}_\theta$ generated by Gaussians were studied by Boca in \cite{bo99-2}. Manin showed that Boca's projections are  quantum theta functions \cite{ma01-1,ma04-2} and a better understanding of these projections is of great relevance for Manin's real multiplication program \cite{ma04-1}. Recently we have presented a time-frequency approach to quantum theta functions in \cite{luma09}. 
\par
The paper is organized as follows: In Section $2$ we present our approach to equivalence bimodules between noncommutative tori and its link with Gabor analysis. We continue with a discussion of Rieffel's projections in noncommutative tori and prove our main results in Section $3$. In the final section we extend the results of Section $3$ to the setting of higher-dimensional noncommutative tori.
%%%%%%%%%%%%%%%%%%%%%%%%%%%%%%%%%%%%%%%%%%%%%%%%%%%%%%%%%%%%%%%%%%%%%%%%%%%%%%%%%%%%%%%%%%%%%%%%%%%%%%%%%%%%%%%%%%%%%%%%%%%%%%%%%%%%%%%%%%%%%%%%%%%%%%%%
\section{Projective modules over noncommutative tori}

In this section we present the construction of projective modules over noncommutative tori \cite{co80,ri88}, its interpretation in terms of Gabor analysis and its extension demonstrated in \cite{lu09}. 
\subsection{Basics on noncommutative tori}
We start with the observation that $z\mapsto\pi(z)$ is a projective representation of $\mathbb{R}^2$ on $L^2(\mathbb{R})$, i.e. we have
\begin{equation}\label{commrel}
  \pi(z)\pi(z')=e^{2\pi ix\cdot\eta}\pi(z+z')~~\text{for}~~z=(x,\omega),z'=(y,\eta)~~\text{in}~~\mathbb{R}^2.
\end{equation} 
We denote the $2$-cocycle in the preceding equation by $c(z,z')=e^{2\pi ix\cdot\eta}$.
The relation in \eqref{commrel} relies on the canonical commutation relation for $M_\omega$ and $T_x$:
\begin{equation}\label{TMcommrel}
  M_\omega T_x=e^{2\pi ix\cdot\omega}T_xM_\omega~~\text{for}~~z=(x,\omega)\in\mathbb{R}^2.
\end{equation}
An application of \eqref{TMcommrel} to the left-hand side of \eqref{commrel} gives a commutation relation for time-frequency shifts:
\begin{equation}\label{commrelsymp}
  \pi(z)\pi(z')=c_{\mathrm{symp}}(z,z')\pi(z')\pi(z),~~z=(x,\omega),z'=(y,\eta)\in\mathbb{R}^2,
\end{equation}
where $c_{\mathrm{symp}}(z,z')=c(z,z')\overline{c(z',z)}=e^{2\pi i(y\cdot\omega-x\cdot\eta)}$ denotes the symplectic bicharater. The term in the exponential of $c_{\mathrm{symp}}$ is the standard symplectic form $\Omega$ of $z=(x,\omega)$ and $z'=(y,\eta)$.  
\par
For our purpose it is useful to view the noncommutative torus $\mathcal{A}_\theta$ as twisted group $C^*$-algebra $C^*(\Lambda,c)$ of a lattice $\Lambda$ in $\mathbb{R}^2$. Recall that $C^*(\Lambda,c)$ is the enveloping $C^*$-algebra of the involutive twisted group algebra $\ell^1(\Lambda,c)$, which is $\ell^1(\Lambda)$ with {\it twisted convolution} $\natural$ as multiplication and $^*$ as involution. More precisely, let ${\bf a}=(a(\lambda))_\lambda$ and ${\bf b}=(b(\lambda))_\lambda$ be in $\ell^1(\Lambda)$. Then the {\it twisted convolution} of ${\bf a}$ and ${\bf b}$ is defined by
\begin{equation}
  {\bf a}\natural{\bf b}(\lambda)=\sum_{\mu\in\Lambda}a(\mu)b(\lambda-\mu)c(\mu,\lambda-\mu)~~\text{for}~~\lambda,\mu\in\Lambda,
\end{equation} 
and {\it involution} ${\bf a}^*=\big(a^*(\lambda)\big)$ of ${\bf a}$ is given by
\begin{equation}
  a^*(\lambda)=\overline{c(\lambda,\lambda)a(-\lambda)}~~\text{for}~~\lambda\in\Lambda.
\end{equation}
Let $\Lambda$ be a lattice in $\mathbb{R}^2$. Then the restriction of the projective representation to $\Lambda$ in $\mathbb{R}^2$ gives that  $\lambda\mapsto\pi(\lambda)$ is a projective representation of $\Lambda$ on $\ell^2(\Lambda)$. Furthermore, this projective representation of a lattice $\Lambda$ in $\mathbb{R}^{2}$ gives a non-degenerate involutive representation of $\ell^1(\Lambda,c)$ on $\ell^2(\Lambda)$ by
\begin{equation*}
  \pi_\Lambda({\bf a}):=\sum_{\lambda\in\Lambda}a(\lambda)\pi(\lambda)~~\text{for}~~{\bf a}=(a(\lambda))\in\ell^1(\Lambda),
\end{equation*}
i.e. $\pi_\Lambda({\bf a}\natural{\bf b})=\pi_\Lambda({\bf a})\pi_\Lambda({\bf b})$ and $\pi_\Lambda(a^*)=\pi_\Lambda({\bf a})^*$. Moreover, this involutive representation of $\ell^1(\Lambda,c)$ is {\it faithful}: $\pi_\Lambda({\bf a})=0$ implies ${\bf a}=0$ for ${\bf a}\in\ell^1(\Lambda)$, see e.g. \cite{ri88}. 
\par
In Rieffel's classification of projective modules over noncommutative tori \cite{ri88} a key insight was the relevance of a lattice $\Lambda^\circ$ associated to $\Lambda$:  
\begin{equation}
  \Lambda^\circ=\{(x,\omega)\in\mathbb{R}^2: c_{\mathrm{symp}}\big((x,\omega),\lambda\big)=1~~\text{for all}~~\lambda\in\Lambda\}
\end{equation}
or equivalently by
\begin{equation}
  \Lambda^\circ=\{z\in\mathbb{R}^2:\pi(\lambda)\pi(z)=\pi(z)\pi(\lambda)~~\text{for all}~~\lambda\in\Lambda\}.
\end{equation}
Following Feichtinger and Kozek we call $\Lambda^\circ$ the {\it adjoint lattice} \cite{feko98}. The lattices $\Lambda$ and $\Lambda^\circ$ are the key players in the duality theory of Gabor analysis, i.e. the Janssen representation of Gabor frames, Wexler-Raz biorthogonality relations and the Ron-Shen duality principle \cite{dalala95,feko98,ja95,rosh97}. 
\par
In the following we want to study weighted analogues of the twisted group algebra. For $s\ge 0$ let $\ell^1_s(\Lambda)$ be the space of all sequences ${\bf a}$ with $\|{\bf a}\|_{\ell^1_s}=\sum |a(\lambda)|(1+|\lambda|^2)^{s/2}<\infty$ we consider $(\ell^1_s(\Lambda),\natural,*)$. More explicitly, 
\begin{equation*}
  \mathcal{A}^1_s(\Lambda,c)=\{A\in\mathcal{B}(L^2(\mathbb{R})):A=\sum_{\lambda}a(\lambda)\pi(\lambda),\|{\bf a}\|_{\ell^1_s}<\infty\}
\end{equation*}
is an involutive Banach algebra with respect to the norm $\|A\|_{\mathcal{A}^1_s(\Lambda)}=\sum_{\lambda}|a(\lambda)|(1+|\lambda|^2)^{s/2}$. Note that $\mathcal{A}^1_s(\Lambda,c)$ is a dense subalgebra of ${C^*(\Lambda,c)}$. 
%We call $\mathcal{A}^1_s(\Lambda,c)$ a (weighted) {\it noncommutative Wiener algebra} because it is the noncommutative analogue of Wiener's algebra of Fourier series with absolutely convergent Fourier coefficients. 
The {\it smooth noncommutative torus} $\mathcal{A}^{\infty}(\Lambda,c)=\bigcap_{s\ge 0}\mathcal{A}^1_{s}(\Lambda,c)$ and $\mathcal{A}^{\infty}(\Lambda,c)$ is an involutive Frechet algebra with respect to $\natural$ and $\ast$ whose topology is defined by a family of submultiplicative norms $\{\|.\|_{\mathcal{A}^1_{s}}|s\ge 0\}$: 
\begin{equation*}
  \|A\|_{\mathcal{A}^1_{s}}=\sum_{\lambda\in\Lambda}|a(\Lambda)|(1+|\lambda|^2)^{s/2}~~\text{for}~~A\in\mathcal{A}^\infty_s(\Lambda,c).
\end{equation*} 
In other words $\mathcal{A}^{\infty}(\Lambda,c)$ is the the image of ${\bf a}\mapsto\pi_\Lambda({\bf a})$ for ${\bf a}\in\mathscr{S}(\Lambda)$, where $\mathscr{S}(\Lambda)$ denotes the space of rapidly decreasing sequences on $\Lambda$. The smooth noncommutative torus $\mathcal{A}^{\infty}(\Lambda,c)$ is the prototype example of a noncommutative manifold \cite{co80,co94-1}. 
\par
Recall that a unital subalgebra $\mathcal{A}$ of a unital $C^*$-algebra $\mathcal{B}$ with common unit is called {\it spectrally invariant}, if for  $A\in\mathcal{A}$ with $A^{-1}\in\mathcal{B}$ one actually has that $A^{-1}\in\mathcal{A}$. 
\begin{proposition}\label{SpectInvNCT}
Let $\Lambda$ be a lattice in $\mathbb{R}^2$. Then $\mathcal{A}^1_s(\Lambda,c)$ and $\mathcal{A}^\infty(\Lambda,c)$ are spectrally invariant subalgebras of $C^*(\Lambda,c)$. Consequently, $\mathcal{A}^1_s(\Lambda,c)$ and $\mathcal{A}^\infty(\Lambda,c)$ are invariant under holomorphic function calculus.
\end{proposition}
The spectral invariance of $\mathcal{A}^\infty(\Lambda,c)$ in  $C^*(\Lambda,c)$ was demonstrated by Connes in \cite{co80} and the case of $\Lambda=\alpha\mathbb{Z}\times\beta\mathbb{Z}$ for rational lattice constants $\alpha$ and $\beta$ was rediscovered by Janssen in the content of Gabor analysis \cite{ja95}. The connection between the work of Connes and Janssen was pointed out in \cite{lu06}. The extension of Janssen's result to lattices with irrational lattice constants was the motivation of Gr\"ochenig and Leinert to prove the spectral invariance of $\mathcal{A}^1_s(\Lambda,c)$ in $C^*(\Lambda,c)$ in \cite{grle04}, see also \cite{gr07}. 
\subsection{Modulation spaces and Hilbert $C^*(\Lambda,c)$-modules}

The construction of Hilbert $C^*(\Lambda,c)$-modules is based on a class of function spaces introduced by Feichtinger in \cite{fe83-4}, the so-called {\it modulation spaces}. In the last two decades modulation spaces have found many applications in harmonic analysis and time-frequency analysis, see the interesting survey article \cite{fe06} for an extensive bibliography. We briefly recall the definition and basic properties of a special class of modulation spaces, $M^1_s(\mathbb{R})$, since these provide the correct framework for our investigation.
\par
If $g$ is a window function in $L^2(\mathbb{R})$, then the {\it short-time Fourier transform} (STFT) of a function or distribution $f$ is defined by 
\begin{equation}
  V_gf(x,\omega)=\langle f,\pi(x,\omega)g\rangle=\int_{\mathbb{R}}f(t)\overline{g}(t-x)e^{-2\pi ix\cdot\omega}dt.
\end{equation}
The STFT $V_gf$ of $f$ with respect to the window $g$ measures the time-frequency content of a function $f$. Modulation spaces are classes of function spaces, where the norms are given in terms of integrability or decay conditions of the STFT. 
%In the present note we restrict our interest to the modulation spaces $M^1_s(\mathbb{R})$ for the submultiplicative weight $v_s(x,\omega)=(1+|x|^2+|\omega|^2)^{s/2}$. 
\par
If the window function is the Gaussian $\varphi(t)=e^{-\pi t^2}$, then the modulation space $M^1_s(\mathbb{R})$ is the space 
\begin{equation*}
  M^1_s(\mathbb{R})=\{f\in L^2(\mathbb{R}): \|f\|_{M^1_s}:=\int_{\mathbb{R}}|V_\varphi f(x,\omega)|(1+|x|^2+|\omega|^2)^{s/2}dxd\omega<\infty\}.
\end{equation*}
The space $M^1_0(\mathbb{R})$ is the well-known Feichtinger algebra, which was introduced in \cite{fe81-3} as the minimal strongly character invariant Segal algebra and is often denoted by  $S_0(\mathbb{R})$. In time-frequency analysis the modulation space $M^1_s(\mathbb{R})$ has turned out to be a good class of windows for Gabor frames, pseudo-differential operators and time-varying channels. In \cite{lu07,lu09} we emphasized that these function spaces provide a convenient class of pre-equivalence $C^*(\Lambda,c)$-modules. To link our approach to Rieffel's work we rely on a description of Schwartz's class of test functions $\mathscr{S}(\mathbb{R})$ in terms of the STFT:
\begin{equation*}
  \mathscr{S}(\mathbb{R})=\bigcap_{s\ge 0}M^1_{s}(\mathbb{R})
\end{equation*}
with seminorms $\|f\|_{M^1_{s}}=\|V_gf\|_{L^1_s}$ for $s\ge 0$ and a fixed $g\in M^1_s(\mathbb{R})$.
\par 
The basic fact in Rieffel's construction of projective modules over noncommutative tori and in Gabor analysis is the so-called {\it Fundamental Identity of Gabor Analysis} (FIGA). In \cite{felu06} we have discussed the validity of FIGA for various classes of function spaces. In the present setting we need the FIGA for functions in $M^1_s(\mathbb{R})$ or in $\mathscr{S}(\mathbb{R})$. 
\begin{proposition}[FIGA]\label{FIGA}
Let $\Lambda$ be lattice in $\mathbb{R}^2$. Then for $f,g,h,k\in M^1_s(\mathbb{R})$ or in $\mathscr{S}(\mathbb{R})$ the following identity holds:
\begin{equation}\label{FIGA-eq}
  \sum_{\lambda\in\Lambda}\langle f,\pi(\lambda)g\rangle\langle \pi(\lambda)h,k\rangle={\mathrm{vol}(\Lambda)}^{-1}\sum_{\lambda^\circ\in\Lambda^\circ}\langle f,\pi(\lambda^\circ)k\rangle\langle\pi(\lambda^\circ)h, g\rangle,
\end{equation}
where $\mathrm{vol}(\Lambda)$ denotes the volume of a fundamental domain of $\Lambda$.
\end{proposition}
Note, that $\lambda\mapsto\pi(\lambda)$ and $\lambda^\circ\mapsto\pi(\lambda^\circ)$ are reducible projective representations of $\Lambda$ and $\Lambda^\circ$, respectively. Therefore the FIGA expresses a relation between the matrix coefficients of these reducible projective representations: 
\begin{equation}
  \langle V_{g}f(\lambda),V_{h}k(\lambda)\rangle_{\ell^2(\Lambda)}={\mathrm{vol}(\Lambda)}^{-1}\langle V_{k}f,V_{h}g\rangle_{\ell^2(\Lambda^\circ)}.
\end{equation}
Therefore one has to impose some extra conditions to get Schur-type orthogonality relations. This fact underlies the Wexler-Raz biorthogonality relations, which we discuss in the following section. 
\par
To motivate the left and right actions of the noncommutative torus on $M^1_s(\mathbb{R})$ or $\mathscr{S}(\mathbb{R})$ we write FIGA in the following form:
\begin{equation}
\Big\langle \sum_{\lambda\in\Lambda}\langle f,\pi(\lambda)g\rangle\pi(\lambda)h,k\Big\rangle=\Big\langle{\mathrm{vol}(\Lambda)}^{-1}\sum_{\lambda^\circ}\pi(\lambda^\circ)^*f\langle \pi(\lambda^\circ)^*g,h\rangle,k\Big\rangle
\end{equation}
The preceding equation indicates a left action of $\mathcal{A}^1_s(\Lambda,c)$ and a right action of $\mathcal{A}^1_s(\Lambda^\circ,\overline{c})$ on functions $g\in M^1_s(\mathbb{R})$ by
\begin{align}
  &\pi_{_\Lambda}({\bf a})\cdot g=\sum_{\lambda\in\Lambda}a(\lambda)\pi(\lambda)g &\text{for}~~{\bf a}\in\ell^1(\Lambda),\\
  &\pi_{\Lambda^\circ}({\bf b})\cdot g={\mathrm{vol}(\Lambda)}^{-1}\sum_{\lambda^\circ\in\Lambda^\circ}\pi(\lambda^\circ)^*g\,\overline{b(\lambda^\circ)} &\text{for}~~{\bf b}\in\ell^1(\Lambda^\circ),
\end{align}
and additionally the $\mathcal{A}^1_s(\Lambda,c)$-valued inner product ${_\Lambda}\langle.,.\rangle$ and $\mathcal{A}^1_s(\Lambda^\circ,\overline{c})$-valued inner product $\langle.,.\rangle_{\Lambda^\circ}$
\begin{align}
 &{_\Lambda}\langle f,g\rangle=\sum_{\lambda\in\Lambda}\langle f,\pi(\lambda)g\rangle\pi(\lambda)\\
 &\langle f,g\rangle_{\Lambda^\circ}={\mathrm{vol}(\Lambda)}^{-1}\sum_{\lambda^\circ\in\Lambda^\circ}\pi(\lambda^\circ)^*\langle\pi(\lambda^*f,g)\rangle,
\end{align}
for $f,g\in M^1(\mathbb{R})$. 
Consequently, we have that ${_\Lambda}\langle f,g\rangle$ and $\langle f,g\rangle_{\Lambda^\circ}$ is an element of $\mathcal{A}^1_s(\Lambda,c)$. The crucial observation is that ${_\Lambda}\langle f,g\rangle$ is a $\mathcal{A}^1_s(\Lambda,c)$-valued  inner product. In \cite{lu09} we have demonstrated that $M^1_s(\mathbb{R})$ becomes a full left Hilbert $C^*(\Lambda,c)$-module ${_\Lambda}V$ when completed with respect to the norm ${_\Lambda}\|f\|=\|{_\Lambda}\langle f,f\rangle\|^{1/2}$ for $f\in M^1_s(\mathbb{R})$. 
\par
In addition we have an analogous result for the opposite $C^*$-algebra of $C^*(\Lambda,c)$, i.e. $C^*(\Lambda^\circ,\overline{c})$. $M^1_s(\mathbb{R})$ becomes a full right Hilbert $C^*(\Lambda^\circ,\overline{c})$-module $V_{\Lambda^\circ}$ for the right action of $\mathcal{A}^1_s(\Lambda^\circ,\overline{c})$ on $M^1_s(\mathbb{R})$
when completed with respect to the norm $\|f\|_{\Lambda^\circ}=\|\langle f,f\rangle_{\Lambda^\circ}\|_{\text{op}}^{1/2}$. 
\par
Most notably the $C^*$-valued inner products ${_\Lambda}\langle.,.\rangle$ and $\langle.,.\rangle_{\Lambda^\circ}$ satisfy Rieffel's associativity condition:
\begin{equation}\label{eq:AssCond}
  {_\Lambda}\langle f,g\rangle\cdot h=f\cdot\langle g,h\rangle_{\Lambda^\circ},~~f,g,h\in M^1_s(\mathbb{R})
\end{equation}   
 The identity \eqref{eq:AssCond} is equivalent to 
 \begin{equation*}
   \big\langle{_\Lambda}\langle f,g\rangle\cdot h,k\big\rangle=\big\langle f\cdot\langle g,h\rangle_{\Lambda^\circ},k\big\rangle
 \end{equation*}
 for all $k\in M^1_s(\mathbb{R})$. More explicitly, the associativity condition reads as follows
 \begin{equation*}
   \sum_{\lambda\in\Lambda}\langle f,\pi(\lambda)g\rangle\langle\pi(\lambda)h,k\rangle={\mathrm{vol}(\Lambda)}^{-1}\sum_{\lambda^\circ\in\Lambda^\circ}\langle f,\pi(\lambda^\circ)k\rangle\langle \pi(\lambda^\circ)h,g\rangle.
 \end{equation*}
In other words, the associativity condition is the Fundamental Identity of Gabor analysis.

\par
Furthermore, we have that $M^1_s(\mathbb{R})$ is a right pre-inner product module over $\mathcal{A}^1_s(\Lambda^\circ,\overline{c})$ for the {\it adjoint lattice} $\Lambda^\circ$ of $\Lambda$.  Consequently, we get that ${_\Lambda}V_{\Lambda^\circ}$ is an equivalence bimodule between $C^*(\Lambda,c)$ and $C^*((\Lambda^\circ,\overline{c}))$. By a result of Connes we have that $M^1_s(\mathbb{R})$ is an equivalence bimodule between $\mathcal{A}^1_s(\Lambda,c)$ and $\mathcal{A}^1_s(\Lambda^\circ,\overline{c})$. We summarize these observations and result in the following theorem, which is a special case of the main result in \cite{lu09} and provides the setting for our investigation.
\begin{theorem}
Let $\Lambda$ be a lattice in $\mathbb{R}^2$. For any $s\ge 0$ we have that $M^1_s(\mathbb{R})$ is an equivalence bimodule between $\mathcal{A}^1_s(\Lambda,c)$ and $\mathcal{A}^1_v(\Lambda^\circ,\overline{c})$, and $\mathscr{S}(\mathbb{R})$ is an equivalence bimodule between $\mathcal{A}^\infty(\Lambda,c)$ and $\mathcal{A}^\infty(\Lambda^\circ,\overline{c})$. Consequently, $M^1_s(\mathbb{R})$ is a finitely generated projective left $\mathcal{A}^1_s(\Lambda^\circ,\overline{c})$-module and $\mathscr{S}(\mathbb{R})$ is a finitely generated projective left $\mathcal{A}^1_s\Lambda^\circ,\overline{c})$-module. 
\end{theorem}
The statement about $\mathscr{S}(\mathbb{R})$ was proved by Connes in \cite{co80}. Another way of expressing the content of the preceding theorem is that $\mathcal{A}^1_s(\Lambda,c)$ and $\mathcal{A}^1_s(\Lambda^\circ,\overline{c})$ are Morita-Rieffel equivalent and also $\mathcal{A}^\infty(\Lambda,c)$ and $\mathcal{A}^\infty(\Lambda^\circ,\overline{c})$ are Morita-Rieffel equivalent.
%%%%%%%%%%%%%%%%%%%%%%%%%%%%%%%%%%%%%%%%%%%%%%%%%%%%%%%%%%%%%%%%%%%%%%%%%%%%%%%%%%%%%%%%%%%%%%%%%%%%%%%%%%%%%%%%%%%%%%%%%%%%%%%%%%%%%%%%%%%%%%%%%%%%%%%
\section{Projections in noncommutative tori}

In this section we revisit the construction of projections in $C^*(\Lambda,c)$ presented in \cite{ri81} in terms of Gabor analysis. We start with some observations on $C^*(\Lambda,c)$-module rank-one operator on ${_\Lambda}V$, i.e operators of the form:  
\begin{equation*}
  \Theta_{g,h}^\Lambda f={_\Lambda}\langle f,h\rangle\cdot h=\sum_{\lambda\in\Lambda}\langle f,\pi(\lambda)g\rangle\pi(\lambda)h~~\text{for}~~f,g,h\in {_\Lambda}V.
\end{equation*}
The operators $\Theta_{g,h}^\Lambda$ are adjointable operators on ${_\Lambda}V$, i.e. ${_\Lambda}\langle\Theta_{g,h}^\Lambda f,k\rangle={_\Lambda}\langle f,\Theta_{h,g}^\Lambda k\rangle$. %In the case of $\Theta_{g,h}^\Lambda$ the adjoint operator $(\Theta_{g,h}^\Lambda)^*=\Theta_{h,g}^\Lambda$. 
Since ${_\Lambda}V$ is a finitely generated projective $C^*(\Lambda,c)$-module, every adjointable operator on ${_\Lambda}V$ is a finite sum of rank-one operators $\Theta_{g,h}^\Lambda$, i.e. a finite rank $C^*(\Lambda,c)$-module operator. 
\par
We collect some elementary observation on projections in $C^*(\Lambda,c)$, i.e. operators $P$ such that $P=P^*=P^2$.
\begin{lemma}
Let $g,h$ be in ${_\Lambda}V$ with $\|g\|_{\Lambda}=1$. Then the following holds:
\begin{enumerate}
  \item[(a)] $\Theta_{g,g}^\Lambda$ and $\Theta_{h,h}^\Lambda$ are selfadjoint projections and $\Theta_{g,h}^\Lambda$ is a partial isometry.
  \item[(b)] If $\|g-h\|_\Lambda<1/2$, then there exists a unitary adjointable ${_\Lambda}V$ module operator $U$ such that $Ug=h$ and therefore  $\Theta_{g,g}^\Lambda$ and $\Theta_{h,h}^\Lambda$ are unitarily equivalent.
\end{enumerate} 
\end{lemma}
\begin{proof}
Assertion (a) can be deduced from a series of elementary computations. Assertion (b) may be derived from the fact that $\Theta_{g,g}^\Lambda$ and $\Theta_{h,h}^\Lambda$ are unitarily equivalent if $\|\Theta_{g,g}^\Lambda-\Theta_{h,h}^\Lambda\|_\Lambda<1$ and the following inequalities:
\begin{equation*}
\|\Theta_{g,g}^\Lambda-\Theta_{h,h}^\Lambda\|_\Lambda\le\|\Theta_{g,g}^\Lambda-\Theta_{g,h}^\Lambda\|_\Lambda+\|\Theta_{g,h}^\Lambda-\Theta_{h,h}^\Lambda\|_\Lambda\le 2\|g-h\|_\Lambda.
\end{equation*}
Finally we want to describe the unitary module operators for ${_\Lambda}V$, i.e. those $U$ such that ${_\Lambda}\langle Uf,g\rangle={_\Lambda}\langle f,Ug\rangle$. More explicitly, this means that $U$ is a unitary operator on $L^2(\mathbb{R})$ such that $\pi(\lambda)U=U\pi(\lambda)$ for all $\lambda\in\Lambda$. In \cite{feko98} $U$ such operators are called $\Lambda$-invariant.
\end{proof}
Note that we have for $f,g$ that $\|f-g\|_\Lambda^2\le\|V_{f-g}(f-g)\|_{\ell^1_s}$ by
\begin{equation*}
  \|f\|_\Lambda^2\le\sum_{\lambda\in\Lambda}|V_ff(\lambda)|(1+|\lambda|^2)^{s/2}.
\end{equation*}
\par
In our setting we actually have an equivalence bimodule ${_\Lambda}V_{\Lambda^\circ}$ between $C^*(\Lambda,c)$ and $C^*(\Lambda^\circ,\overline{c})$ that provides an additional form to express under which conditions $g\in {_\Lambda}V_{\Lambda^\circ}$ yields a projection ${_\Lambda}\langle g,g\rangle$ in $C^*(\Lambda,c)$ as pointed out in \cite{ri81}. 
\begin{lemma}
  Let $g$ be in ${_\Lambda}V_{\Lambda^\circ}$. Then $P_g:={_\Lambda}\langle g,g\rangle$ is a projection in $C^*(\Lambda,c)$ if and only if $g\langle g,g\rangle_{\Lambda^\circ}=g$. If $g\in M^1_s(\mathbb{R})$ or $\mathscr{S}(\mathbb{R})$, then $P_g$ gives a projection in $\mathcal{A}^1_s(\Lambda,c)$ or $\mathcal{A}^\infty(\Lambda,c)$, respectively. 
\end{lemma}
\begin{proof}
 First we assume that $g\langle g,g\rangle{_{\Lambda^\circ}}=g$ for some $g$ in ${_\Lambda}V_{\Lambda^\circ}$. Then we have that
 \begin{equation*}
  P_g^2={_\Lambda}\langle g,g\rangle{_\Lambda}\langle g,g\rangle={_\Lambda}\big\langle {_\Lambda}\langle g,g\rangle g,g\big\rangle={_\Lambda}\langle g\langle g,g\rangle_{\Lambda^\circ},g\rangle={_\Lambda}\langle g,g\rangle=P_g
 \end{equation*}
 and $P_g^*=P_g$.
 \par
 Now we suppose that ${_\Lambda}\langle g,g\rangle$ is a projection in $C^*(\Lambda,c)$. Then the following elementary computation yields the assertion: 
 \begin{eqnarray*} 
  {_\Lambda}\big\langle g\langle g,g\rangle{_{\Lambda^\circ}}-g,g\langle g,g\rangle{_{\Lambda^\circ}}-g\big\rangle&=&{_\Lambda}\big\langle {_\Lambda}\langle g,g\rangle g-g,{_\Lambda}\langle g,g\rangle g-g\big\rangle\\
  &=&{_\Lambda}\big\langle {_\Lambda}\langle g,g\rangle g,{_\Lambda}\langle g,g\rangle g\big\rangle -{_\Lambda}\big\langle g,{_\Lambda}\langle g,g\rangle g\big\rangle\\
  &-&{_\Lambda}\big\langle{_\Lambda}\langle g,g\rangle g,g\big\rangle+{_\Lambda}\langle g,g\rangle=0.
 \end{eqnarray*}
 In the case that $g\in M^1_s(\mathbb{R})$ or $\mathscr{S}(\mathbb{R})$, then the condition $g\langle g,g\rangle{_{{\Lambda^\circ}}}=g$ holds in $g\in M^1_s(\mathbb{R})$ or $\mathscr{S}(\mathbb{R})$. Consequently the preceding computations remain valid in $\mathcal{A}^1_s(\Lambda,c)$ or $\mathcal{A}^\infty(\Lambda,c)$.
\end{proof}
There is a class of $g$ in $V_{\Lambda^\circ}$ where the condition $g\langle g,g\rangle{_{\Lambda^\circ}}=g$ is fulfilled.  Namely those $g\in V{_{\Lambda^\circ}}$ such that $\langle g,g\rangle{_{\Lambda^\circ}}=1_{\ell^2(\Lambda^\circ)}$. 
We denote the set of all these $g$ the unit sphere $\mathbb{S}(V_{\Lambda^\circ})$ of $V{_{\Lambda^\circ}}$. The unit sphere $\mathbb{S}(V_{\Lambda^\circ})$ has an intrinsic description in terms of Gabor frames and goes by the name of {\it Wexler-Raz biorthogonality relations}.  
\par
The link between the rank-one module operators $\Theta_{g,h}^\Lambda$ and Gabor analysis is the observation that these are the so-called {\it Gabor frame-type operator} $\Theta_{g,h}^\Lambda$ and that $\Theta_{g,g}^\Lambda$ is the Gabor frame operator of the Gabor systems $\mathcal{G}(g,\Lambda)$.
If $\Theta_{g,g}^\Lambda$ is invertible on $L^2(\mathbb{R})$, then $\mathcal{G}(g,\Lambda)$ is a Gabor frame for $L^2(\mathbb{R})$, i.e. there exist $A,B>0$ such that 
\begin{equation*}
 A\|f\|_2^2\le \langle \Theta_{g,g}^\Lambda f,f\rangle_{L^2(\mathbb{R})}=\sum_{\lambda\in\Lambda}|\langle f,\pi(\lambda)g\rangle|^2\le B\|f\|_2^2
\end{equation*}
for all $L^2(\mathbb{R})$. An important consequence of the invertibility of the Gabor frame operator is the existence of discrete expansions for $f\in L^2(\mathbb{R})$: 
\begin{equation}
  f=\Theta_{g,h}^\Lambda f=\sum_{\lambda\in\Lambda}\langle f,\pi(\lambda)g\rangle\pi(\lambda)h
\end{equation}
for some $h\in L^2(\mathbb{R})$, a so-called {\it dual Gabor atom}. Among the various dual Gabor atoms there exists a {\it canonical dual Gabor atom} $\tilde{h}$ that is determined by the equation $(\Theta_{g,g}^\Lambda){h}_0=g$, i.e. ${h}_0=S_{g,\Lambda}^{-1}g$. In the case that $Cg=h_0$ for some constant $C$, then the (dual) Gabor frame $\mathcal{G}(g,\Lambda)$ is a {\it tight Gabor frame} for $L^2(\mathbb{R})$ and $h_0$ is often referred to as {\it tight Gabor atom}. 
\begin{theorem}\label{projGabor}
Let $\mathcal{G}(g,\Lambda)$ be a Gabor system on $L^2(\mathbb{R})$ with $g$ in $M^1_s(\mathbb{R})$ or $\mathscr{S}(\mathbb{R})$. Then $P_g={_\Lambda}\langle g,g\rangle$ is a projection in $\mathcal{A}^1_s(\Lambda,c)$ or $\mathcal{A}^\infty(\Lambda,c)$ if and only if one of the following condition holds:
\begin{enumerate}
 \item[(i)] $\mathcal{G}(g,\Lambda)$ is a tight Gabor frame for $L^2(\mathbb{R})$.
 \item[(ii)] $\mathcal{G}(g,\Lambda^\circ)$ is an orthogonal system.  
 \item[(iii)] $g\in\mathbb{S}(V_{\Lambda^\circ})$.
 \item[(iv)] $\langle g,\pi(\lambda^\circ)g\rangle=\mathrm{vol}(\Lambda)\delta_{\lambda^\circ,0}$ for all $\lambda^\circ\in\Lambda^\circ$.
\end{enumerate}
\end{theorem}
\begin{proof}
Recall that there are traces $\mathrm{tr}_\Lambda$ and $\mathrm{tr}_{\Lambda^\circ}$ on $C^*(\Lambda,c)$ and $C^*(\Lambda^\circ,\overline{c})$, where $\mathrm{tr}_\Lambda(A)=a_0$ and $\mathrm{tr}_{\Lambda^\circ}(B)=\mathrm{vol}(\Lambda)^{-1}b_0$ for $A=\sum a(\lambda)\pi(\lambda)$ and $B=\sum b(\lambda^\circ)\pi(\lambda^\circ)$.
  \begin{item} 
    \item[(i)$\Leftrightarrow$(ii)] The assumption $g,h$ in $M^1_s(\mathbb{R})$ implies the boundedness of the Gabor frame operators $\Theta_{g,g}^\Lambda$ on $L^2(\mathbb{R})$. Furthermore $\Theta_{g,g}^\Lambda$ has a Janssen representation 
\begin{equation}
  \Theta_{g,g}^\Lambda f={_\Lambda}\langle f,g\rangle\cdot g=f\cdot\langle g,g\rangle_{\Lambda^\circ}=\Theta_{f,g}^{\Lambda^\circ}g.
\end{equation}
In other words, the Janssen representation of Gabor frame-type operators is the associativity condition for ${_\Lambda}\langle.,.\rangle$ and $\langle.,.\rangle_{\Lambda^\circ}$. $\mathcal{G}(g,\Lambda)$ is a tight Gabor frame if and only if $\Theta_{g,g}^\Lambda$ is a multiple of the identity operator on $L^2(\mathbb{R})$ if and only if $\mathcal{G}(g,\Lambda)$ is an orthogonal system.
  \item[(ii)$\leftrightarrow$(iii)] This is just a reformulation of (i) in terms of the ${_\Lambda}\langle.,.\rangle$ inner product 
  \item[(iii)$\leftrightarrow$(iv)] By taking the trace of the assertion (iii) we get that
\begin{equation*} 
   \mathrm{tr}_{\Lambda}({_\Lambda}\langle g,g\rangle)= \langle g,g\rangle=\mathrm{vol}(\Lambda)^{-1}\delta_{\lambda^\circ,0}.
\end{equation*}   
  \end{item}
\end{proof}
The equivalence between (i) and (iv) goes by the name of {Wexler-Raz biorthogonality relations}. In the case of finite-dimensional Gabor frames this result was formulated by the engineers Raz and Wexler in \cite{rawe90}. The extension to the infinite-dimensional case was undertaken by several researchers \cite{dalala95,ja95,rosh97} and led to the duality theory of Gabor frames. We followed the approach of Janssen to duality theory. The Wexler-Raz biorthogonality condition may be considered as a Schur-type orthogonalization relation for the reducible representation $\pi_\Lambda$ of $\Lambda$ since it forces the representation $\pi_{\Lambda^\circ}$ of $\Lambda^\circ$ to be a multiple of the trivial representation. 
\par
In the preceding theorem we demonstrated that $g\in\mathbb{S}(V_{\Lambda^\circ})$ is equivalent to the tightness of the Gabor frame $\mathcal{G}(g,\Lambda)$. As noted before there is a canonical tight Gabor frame $\mathcal{G}(h_0,\Lambda)$ for $\tilde{g}=(\Theta_{g,g}^\Lambda)^{-1/2}g$. Janssen and Strohmer have shown in \cite{jast02} that the canonical tight Gabor atom has the following characterization:
Let $\mathcal{G}(g,\Lambda)$ be a Gabor frame for $L^2(\mathbb{R})$. Then the canonical tight Gabor atom $\tilde{g}$ minimizes $\|\tilde{g}-h\|_2$ among all $h$ generating a normalized tight Gabor frame. Note that $\mathrm{tr}_\Lambda({_\Lambda}\langle f,g\rangle)=\langle f,g\rangle$, i.e. $\|\tilde{g}-h\|_2^2=\mathrm{tr}_\Lambda({_\Lambda}\langle \tilde{g}-h,\tilde{g}-h\rangle)$.
\begin{theorem}\label{Gabframes}
Let $\mathcal{G}(g,\Lambda)$ be a Gabor frame for $L^2(\mathbb{R})$. If $g$ is in $M^1_s(\mathbb{R})$ or in $\mathscr{S}(\mathbb{R})$, then ${_\Lambda}\langle {g},\tilde{g}\rangle$ is a projection in $\mathcal{A}^1_s(\Lambda,c)$ or in $\mathcal{A}^\infty(\Lambda,c)$, respectively. Furthermore, $h_0$ minimizes $\|\tilde{g}-h\|_2^2=\mathrm{tr}_\Lambda({_\Lambda}\langle \tilde{g}-h,\tilde{g}-h\rangle)$ among all tight Gabor atoms $h$.
\end{theorem}
\begin{proof}
  By the spectral invariance of $\mathcal{A}^1(\Lambda^\circ,\overline{c})$ and $\mathcal{A}^\infty(\Lambda^\circ,\overline{c})$ in $C^*(\Lambda^\circ,c)$ and by the Janssen representation of the Gabor frame operator $\Theta_{g,g}^\Lambda$ we get $(\Theta_{g,g}^\Lambda)^{-1/2}$ in $\mathcal{A}^1(\Lambda^\circ,\overline{c})$ and $\mathcal{A}^\infty(\Lambda^\circ,\overline{c})$, respectively. Consequently $(\Theta_{g,g}^\Lambda)^{-1/2}g$ is in $M^1_s(\mathbb{R})$ and $\mathscr{S}(R)$, respectively. Observe that $\tilde{g}\in\mathbb{S}(V_{\Lambda^\circ})$ and an application of the preceding theorem yields the desired assertion.
\end{proof}
Before we are able to draw some consequences on projections in noncommutative tori we have to recall well-known results about Gabor systems for the following Gabor atoms $g_1,g_2,g_3$, where $g_1(t)=2^{1/4}e^{-\pi t^2}$ is the Gaussian, $g_2(t)=(\frac{\pi}{2})^{1/2}\frac{1}{\cosh(\pi t)}$ the hyperbolic secant and $g_3(t)=e^{-|t|}$ the two-sided exponential. 
\begin{proposition}\label{JanssenSeip}
The Gabor systems $\mathcal{G}(g_1,\mathbb{Z}\times\theta\mathbb{Z})$, $\mathcal{G}(g_2,\mathbb{Z}\times\theta\mathbb{Z})$ and $\mathcal{G}(g_3,\mathbb{Z}\times\theta\mathbb{Z})$ are Gabor frames for $L^2(\mathbb{R})$ if and only if $\theta<1$.
\end{proposition}
Lyubarskij and Seip proved the result for the Gaussian $g_1$ in \cite{ly92,se92-1}. The statement for $g_2$ was obtained by Janssen and Strohmer in \cite{jast02-1}. Later Janssen was able to settle the case $g_3$ in \cite{ja03-1}. 
\par
The main result allows us to link the existence of Gabor frames to the construction of projections in noncommutative tori, which is based on the seminal contribution of Janssen in \cite{ja95}. Namely the Janssen representation of Gabor operators, i.e. the associativity condition for the noncommutative tori valued inner products, turns the problem of the construction of Gabor frames into a problem about the invertibility of operators in $C^*(\Lambda^\circ,\overline{c})$. Following Janssen's work \cite{ja95} Gr\"ochenig and Leinert interpreted Janssen's result in terms of spectral invariant subalgebras of $C^*(\Lambda^\circ,\overline{c})$ \cite{grle04}. In the following theorem we show that results in Gabor analysis provide a way to smooth projections in noncommutative tori and we give an example of a function, $g_3$, that gives not a projection in the smooth noncommutative torus. Namely $g_3$ is in Feichtinger's algebra $M^1(\mathbb{R})$ but not in $\mathscr{S}(\mathbb{R})$.
%In \cite{lu07,lu09} we were linked this line of research with projective modules over noncommutative tori, in particular that these are multi-window Gabor frames. In the following theorem we continue this line of research and use it to construct projections in smooth subalgebras of noncommutative tori. 
%By Theorem \ref{Gabframes} we get projections in $\mathcal{A}_\theta$, i.e. in $C^*(\mathbb{Z}\times\theta\mathbb{Z},c)$. 
\par
For the sake of simplicity we denote $\mathcal{A}^1(\mathbb{Z}\times\theta\mathbb{Z},c)$ and $\mathcal{A}^\infty(\mathbb{Z}\times\theta\mathbb{Z},c)$ by $\mathcal{A}^1_\theta$ and $\mathcal{A}^\infty_\theta$. In an analogous manner we denote $\mathcal{A}^1(\tfrac{1}{\theta}\mathbb{Z}\times\mathbb{Z},\overline{c})$ and $\mathcal{A}^\infty(\tfrac{1}{\theta}\mathbb{Z}\times\mathbb{Z},\overline{c})$ by $\mathcal{A}^1_{1/\theta}$ and $\mathcal{A}^\infty_{1/\theta}$. Furthermore we abbreviate the $C^*(\mathbb{Z}\times\theta\mathbb{Z},c)$-valued inner product by ${_\theta}\langle.,.\rangle$.
\begin{theorem}
Let $g_1$ be the Gaussian, $g_2$ be the hyperbolic secant and $g_3$ the one-sided exponential. Then 
${_\theta}\langle \tilde{g}_1,\tilde{g}_1\rangle$ and ${_\theta}\langle\tilde{g}_2,\tilde{g}_2\rangle$ are projections in $\mathcal{A}^\infty_\theta$ if and only if $\theta<1$. Furthermore ${_\theta}\langle\tilde{g}_1,\tilde{g}_1\rangle$ are projections in $\mathcal{A}^1_\theta$ if and only if $\theta<1$
\end{theorem}
\begin{proof}
Note that $g_1,g_2$ are elements of $\mathscr{S}(\mathbb{R})$. Therefore $\langle g_i,\pi(\lambda)g_i\rangle$ is a sequence of rapid decay for $i=1,2$. By the Janssen representation $S_{g_i,\mathbb{Z}\times\theta\mathbb{Z}}$ is a Gabor frame if and only if $\langle g_i,g_i\rangle_{\Lambda^\circ}$ is invertible in $\mathcal{A}_{1/\theta}$ for $i=1,2$. By the spectral invariance of $\mathcal{A}^\infty_{1/\theta}$ in $\mathcal{A}_{1/\theta}$ we actually have that $\langle g_i,g_i\rangle_{\Lambda^\circ}$ is an element of $\mathcal{A}^\infty_{1/\theta}$ for $i=1,2$. Consequently, ${_\theta}\langle \tilde{g}_1,\tilde{g}_1\rangle$ and ${_\theta}\langle \tilde{g}_2,\tilde{g}_2\rangle$ are projections in $\mathcal{A}^\infty_{1/\theta}$. 
\par
The final assertion is that ${_\theta}\langle \tilde{g}_3,\tilde{g}_3\rangle$ is a projection in $\mathcal{A}^1_\theta$ if and only if $\theta<1$. We have to check that $g_3$ is not a Schwartz function, but it is an element of Feichtinger's algebra $M^1(\mathbb{R})$. An elementary calculation yields that $g_3$ is not in $\mathscr{S}(\mathbb{R})$. The fact that $g_3$ is in $M^1(\mathbb{R})$ can be established in various ways. We want to refer to a result of Okoudjou. In \cite{ok04} he proved that  $g,g^\prime,g^{\prime\prime}\in L^1(\mathbb{R})$ implies that $g\in M^1(\mathbb{R})$. Now straightforward calculations yield that $g_3,g_3^\prime,g_3^{\prime\prime}$ are in $L^1(\mathbb{R})$ and therefore $g_3$ is in $M^1(\mathbb{R})$. Consequently ${_\theta}\langle g_3,g_3\rangle$ is a projection in $\mathcal{A}^1_\theta$ but not in the smooth noncommutative torus $\mathcal{A}^\infty_\theta$ .
\end{proof}
Since $g_1$ and $g_2$ are invariant with respect to the Fourier transform: $\mathcal{F}g_1=g_1,\mathcal{F}g_2=g_2$, the associated projections fit into the framework of Boca in \cite{bo99-2}. Our approach to projections in noncommutative tori $C^*(\Lambda,c)$ provides that ${_\theta}\langle g_1,g_1\rangle$ is invertible for $\theta<1$, which improves the result in \cite{bo99-2} where the invertibility is established for $\theta<0.948$, and on the other hand it shows that this actually characterizes the invertibility of ${_\theta}\langle g_1,g_1\rangle$. Boca's proof relies on a series of results on theta functions that does not allow one to conclude if the result in \cite{bo99-2} holds if and only if $\theta<1$.  
%%%%%%%%%%%%%%%%%%%%%%%%%%%%%%%%%%%%%%%%%%%%%%%%%%%%%%%%%%%%%%%%%%%%%%%%%%%%%%%%%%%%%%%%%%%%%%%%%%%%%%%%%%%%%%%%%%%%%%%%%%%%%%%%%%%%%%%%%%%%%%%%%%%%%%%
\section{Final remarks}
In the preceding section we constructed projections in $\mathcal{A}_\theta$, because in this case we can apply results of Janssen, Lyubarskij and Seip on Gabor frames for $L^2(\mathbb{R})$. The link between tight Gabor frames and projections in noncommutative tori remains valid in the higher-dimensional case. Recall that the higher-dimensional torus $\mathcal{A}_\Theta$ is defined via a $d\times d$ skew-symmetric matrix $\Theta$ instead of the real number $\theta$. Note that $\mathcal{A}_\Theta$ may be considered as twisted group $C^*$-algebra $C^*(\Lambda,c)$ for a lattice $\Lambda$ in $\mathbb{R}^{2d}$. The higher-dimensional variants of $\mathcal{A}^1_s(\Lambda,c)$ and $\mathcal{A}^\infty(\Lambda,c)$ for $\Lambda$ in $\mathbb{R}^{2d}$ are defined as in the two-dimensional case. The higher-dimensional variant of Theorem \ref{projGabor} holds:
\begin{theorem}\label{projGaborHigher}
Let $\mathcal{G}(g,\Lambda)$ be a Gabor system on $L^2(\mathbb{R}^d)$ for $g\in M^1_s(\mathbb{R}^d)$ or $\mathscr{S}(\mathbb{R}^d)$. Then $P_g={_\Lambda}\langle g,g\rangle$ is a projection in $\mathcal{A}^1_s(\Lambda,c)$ or $\mathcal{A}^\infty(\Lambda,c)$ if and only if one of the following condition holds:
\begin{enumerate}
 \item[(i)] $\mathcal{G}(g,\Lambda)$ is a tight Gabor frame for $L^2(\mathbb{R}^d)$.
 \item[(ii)] $\mathcal{G}(g,\Lambda^\circ)$ is an orthogonal system.  
 \item[(iii)] $g\in\mathbb{S}(V{_{\Lambda^\circ}})$.
 \item[(iv)] $\langle g,\pi(\lambda^\circ)g\rangle=\mathrm{vol}(\Lambda)\delta_{\lambda^\circ,0}$ for all $\lambda^\circ\in\Lambda^\circ$.
\end{enumerate}
\end{theorem}
Furthermore we have that the Theorem \ref{Gabframes} holds in the higher-dimensional case. 
\begin{theorem}\label{GabframesHigher}
Let $\mathcal{G}(g,\Lambda)$ be a Gabor frame for $L^2(\mathbb{R}^d)$. If $g$ is in $M^1_s(\mathbb{R}^d)$ or in $\mathscr{S}(\mathbb{R}^d)$, then ${_\Lambda}\langle \tilde{g},\tilde{g}\rangle$ is a projection in $\mathcal{A}^1_s(\Lambda,c)$ or in $\mathcal{A}^\infty(\Lambda,c)$.
\end{theorem} 
A tensor product type argument allows one to extend Lyubarskij-Seip's result to lattices of the form $\alpha_1\mathbb{Z}\times\cdots\times\alpha_n\mathbb{Z}\times\beta_1\mathbb{Z}\times\cdots\times\beta_n\mathbb{Z}$.
\begin{theorem}
Let $g_1(t)=2^{d/4}e^{-\pi t^2}$ and $\Lambda=\alpha_1\mathbb{Z}\times\cdots\times\alpha_n\mathbb{Z}\times\beta_1\mathbb{Z}\times\cdots\times\beta_n\mathbb{Z}$. Then ${_\Lambda}\langle g_1,g_1\rangle$ is invertible if and only if $\alpha_i\beta_i<1$ for all $i=1,...,n$. Consequently ${_\Lambda}\langle g_1,g_1\rangle$ is a projection in $\mathcal{A}^\infty(\Lambda,c)$.
\end{theorem}
The preceding theorem characterizes the existence of quantum theta functions for $C^*(\alpha_1\mathbb{Z}\times\cdots\times\alpha_n\mathbb{Z}\times\beta_1\mathbb{Z}\times\cdots\times\beta_n\mathbb{Z},c)$. We refer the reader to Manin's papers \cite{ma01-1,ma04-1,ma04-2}, Marcolli's book \cite{ma05-2} and \cite{luma09} for the definition, basic properties of quantum theta function, their relevance to problems in number theory, and to their interpretation in terms of Gabor analysis.
\par
All results with the exceptions of those involving the functions $g_1,g_2,g_3$ hold in much greater generality, see \cite{lu09}. Namely, in the case that $\Lambda$ is a lattice in $G\times\widehat{G}$ for $G$ a locally compact abelian group and $\widehat{G}$ its Pontryagin dual of $G$.
%%%%%%%%%%%%%%%%%%%%%%%%%%%%%%%%%%%%%%%%%%%%%%%%%%%%%%%%%%%%%%%%%%%%%%%%%%%%%%%%%%%%%%%%%%%%%%%%%%%%%%%%%%%%%%%%%%%%%%%%%%%%%%%%%%%%%%%%%%%%%%%%%%%%%%%%
%\bibliographystyle{abbrv}
%\bibliography{nhgbib}
%%%%%%%%%%%%%%%%%%%%%%%%%%%%%%%%%%%%%%%%%%%%%%%%%%%%%%%%%%%%%%%%%%%%%%%%%%%%%%%%%%%%%%%%%%%%%%%%%%%%%%%%%%%%%%%%%%%%%%%%%%%%%%%%%%%%%%%%%%%%%%%%%%%%%%%%

\end{document}